\documentclass[11pt]{article}

\usepackage{amssymb}
\usepackage{amsmath}
\usepackage{amsthm}
\usepackage{mathrsfs}
\usepackage{url}

\newtheorem{thm}{Theorem}
\newtheorem{pro}[thm]{Proposition}

\newtheorem{cor}[thm]{Corollary}
\newtheorem{rmk}[thm]{Remark}
\newtheorem{q}{Problem}
\newtheorem*{ack}{Acknowledgement}

\def\V{{\mathbf V}}

\def\uu{\mathbf{u}}
\def\vv{\mathbf{v}}
\def\ww{\mathbf{w}}
\def\rr{\mathbf{r}}
\def\es{\varnothing}

\def\st{^\ast}
\def\ol#1{\overline{#1}}

\def\ig#1{\mathsf{IG}(#1)}
\def\pre{\;|\;}

\renewcommand\leq{\leqslant}
\renewcommand\geq{\geqslant}

\DeclareMathOperator\im{Im} \DeclareMathOperator\kr{Ker}

\begin{document}

\title{\Large\bf\uppercase{A note on maximal subgroups of free idempotent generated semigroups over bands}}

\author{{\large \textsc{Igor Dolinka}}\\[2mm]
\small Department of Mathematics and Informatics, University of Novi Sad,\\
\small Trg Dositeja Obradovi\'ca 4, 21101 Novi Sad, Serbia\\
\small E-mail: dockie@dmi.uns.ac.rs}

\date{}

\maketitle

\begin{abstract}
We prove that all maximal subgroups of the free idempotent generated semigroup over a band $B$ are free for all $B$
belonging to a band variety $\V$ if and only if $\V$ consists either of left seminormal bands, or of right seminormal
bands.  \footnote{The support of the Ministry of
Education and Science of the Republic of Serbia, through Grant No.\ 174019, is gratefully acknowledged.}
\end{abstract}

Let $S$ be a semigroup, and let $E=E(S)$ be the set of its idempotents; in fact, $E$, along with the multiplication
inherited from $S$, is a partial algebra. It turns out to be fruitful to restrict further the domain of the partial
multiplication defined on $E$ by considering only the pairs $e,f\in E$ for which either $ef\in\{e,f\}$ or
$fe\in\{e,f\}$ (i.e.\ $\{ef,fe\}\cap\{e,f\}\neq\es$). Note that if $ef\in\{e,f\}$ then $fe$ is an idempotent, and the
same is true if we interchange the roles of $e$ and $f$. Such unordered pairs $\{e,f\}$ are called \emph{basic pairs}
and their products $ef$ and $fe$ are \emph{basic products}.

The \emph{free idempotent generated semigroup over $E$} is defined by the following presentation:
$$\ig{E} = \langle E\pre e\cdot f=ef\text{ such that }\{e,f\}\text{ is a basic pair}\,\rangle .$$
Here $ef$ denotes the product of $e$ and $f$ in $S$ (which is again an idempotent of $S$), while $\cdot$ stands for the
concatenation operation in the free semigroup $E^+$ (also to be interpreted as the multiplication in its quotient
$\ig{E}$). An important feature of $\ig{E}$ is that there is a natural homomorphism from $\ig{E}$ onto the subsemigroup
of $S$ generated by $E$, and the restriction of $\phi$ to the set of idempotents of $\ig{E}$ is a
basic-product-preserving bijection onto $E$, see e.g.\ \cite{E4,Nam,P2}.

An important background to these definitions is the notion of the \emph{biordered set} \cite{Hi} of idempotents of a
semigroup and its abstract counterpart. The biordered set of idempotents of $S$ is just a partial algebra on $E(S)$
obtained by restricting the multiplication from $S$ to basic pairs of idempotents. In this way we have that if $B$ is a
band (an idempotent semigroup), then, even though there is an everywhere defined multiplication on $E(B)=B$, its
biordered set \cite{E2} is in general still a partial algebra. Another way of treating biordered sets is to consider
them as relational structures $(E(S),\leq^{(l)},\leq^{(r)})$, where the set of idempotents $E(S)$ is equipped by two
quasi-order relations defined by
\begin{align*}
e\leq^{(l)}f & \text{ if and only if }ef=e,\\
e\leq^{(r)}f & \text{ if and only if }fe=e.
\end{align*}
One of the main achievements of \cite{E3,E4,Nam} is the result that the class of biordered sets considered as
relational structures is \emph{axiomatisable}: there is in fact a finite system of formul\ae\ satisfied by biordered
sets such that any set endowed with two quasi-orders satisfying the axioms in question is a biordered set of
idempotents of some semigroup. In this sense we can speak about the free idempotent generated semigroup over a
biordered set $E$. A fundamental fact which justifies the term `free' is that $\ig{E}$ is the free object in the
category of all semigroups $S$ whose biordered set of idempotents is isomorphic to $E$: if $\psi:E\to E(S)$ is any
isomorphism of biordered sets, then it uniquely extends (via the canonical injection of $E$ into $\ig{E}$) to a
homomorphism $\psi':\ig{E}\to S$ whose image is the subsemigroup of $S$ generated by $E(S)$. This is also true if
$\psi$ is a (surjective) homomorphism of biordered sets (taken as relational structures), so that the freeness property
of $\ig{E}$ carries over to even wider categories of semigroups.

In this short note we consider $\ig{B}$, the free idempotent generated semigroup over (the biordered set of) a band
$B$; more precisely, we are interested in the question whether the maximal subgroups of these semigroups are free. It
was conjectured in \cite{McE} that each maximal subgroup of any semigroup of the form $\ig{E}$ is a free group.
Recently, this was disproved \cite{BMM1} (see also \cite{BMM2}), where a certain 72-element semigroup was found whose
biordered set $E$ of idempotents yields a maximal subgroup in $\ig{E}$ isomorphic to $\mathbb{Z}\oplus\mathbb{Z}$, the
rank 2 free abelian group. Here we will see that a particular 20-element regular band suffices for the same purpose. In
fact, as proved by Gray and Ru\v skuc in \cite{GR}, \emph{every} group can be isomorphic to a maximal subgroup of some
$\ig{E}$, while the assumption that the semigroup $S$ with $E=E(S)$ is finite yields a sole restriction that the groups
in question are finitely presented. This puts forward many new questions, one of which is the characterisation of bands
$B$ for which all subgroups of $\ig{B}$ are free.

More specifically, as a first approximation to the latter question, we may ask for a description of all varieties $\V$
of bands with the property that for each $B\in\V$ the maximal subgroups of $\ig{B}$ are free. To facilitate the
discussion, we depict in Fig.\ \ref{lb} the bottom part of the lattice $\mathcal{L}(\mathbf{B})$ of all band varieties,
along with their standard labels (see also \cite[Diagram II.3.1]{Pe-LinS}).

\begin{figure}[ht]\centering
\begin{picture}(140.00,190.00)
\put(70.00,5.00){\circle*{4}} \put(40.00,35.00){\circle*{4}} \put(70.00,35.00){\circle*{4}}
\put(100.00,35.00){\circle*{4}} \put(40.00,65.00){\circle*{4}} \put(70.00,65.00){\circle*{4}}
\put(100.00,65.00){\circle*{4}} \put(10.00,95.00){\circle*{4}} \put(70.00,95.00){\circle*{4}}
\put(130.00,95.00){\circle*{4}} \put(40.00,125.00){\circle*{4}} \put(100.00,125.00){\circle*{4}}
\put(10.00,155.00){\circle*{4}} \put(70.00,155.00){\circle*{4}} \put(130.00,155.00){\circle*{4}}

\put(70.00,5.00){\line(1,1){30.00}} \put(70.00,5.00){\line(-1,1){30.00}} \put(40.00,35.00){\line(1,1){30.00}}
\put(100.00,35.00){\line(-1,1){30.00}} \put(40.00,65.00){\line(1,1){30.00}} \put(100.00,65.00){\line(-1,1){30.00}}
\put(40.00,65.00){\line(0,-1){30.00}} \put(100.00,65.00){\line(0,-1){30.00}} \put(70.00,35.00){\line(0,-1){30.00}}
\put(70.00,95.00){\line(0,-1){30.00}} \put(70.00,35.00){\line(1,1){60.00}} \put(70.00,35.00){\line(-1,1){60.00}}
\put(10.00,95.00){\line(1,1){60.00}} \put(130.00,95.00){\line(-1,1){60.00}} \put(70.00,95.00){\line(1,1){60.00}}
\put(70.00,95.00){\line(-1,1){60.00}} \put(10.00,155.00){\line(1,1){15.00}} \put(130.00,155.00){\line(-1,1){15.00}}
\put(70.00,155.00){\line(1,1){15.00}} \put(70.00,155.00){\line(-1,1){15.00}}

\put(40.00,180.00){\makebox(0,0)[cc]{$\vdots$}} \put(100.00,180.00){\makebox(0,0)[cc]{$\vdots$}}
\put(70.00,40.00){\makebox(0,0)[cb]{\footnotesize $\mathbf{SL}$}} \put(35.00,35.00){\makebox(0,0)[rc]{\footnotesize
$\mathbf{LZ}$}} \put(105.00,35.00){\makebox(0,0)[lc]{\footnotesize $\mathbf{RZ}$}}
\put(70.00,70.00){\makebox(0,0)[cb]{\footnotesize $\mathbf{ReB}$}} \put(35.00,65.00){\makebox(0,0)[rc]{\footnotesize
$\mathbf{LNB}$}} \put(105.00,65.00){\makebox(0,0)[lc]{\footnotesize $\mathbf{RNB}$}}
\put(70.00,100.00){\makebox(0,0)[cb]{\footnotesize $\mathbf{NB}$}} \put(5.00,95.00){\makebox(0,0)[rc]{\footnotesize
$\mathbf{LRB}$}} \put(135.00,95.00){\makebox(0,0)[lc]{\footnotesize $\mathbf{RRB}$}}
\put(70.00,160.00){\makebox(0,0)[cb]{\footnotesize $\mathbf{RB}$}} \put(5.00,155.00){\makebox(0,0)[rc]{\footnotesize
$\mathbf{LSNB}$}} \put(135.00,155.00){\makebox(0,0)[lc]{\footnotesize $\mathbf{RSNB}$}}
\put(35.00,125.00){\makebox(0,0)[rc]{\footnotesize $\mathbf{LQNB}$}}
\put(105.00,125.00){\makebox(0,0)[lc]{\footnotesize $\mathbf{RQNB}$}}
\end{picture}
\caption{The bottom part of the lattice of all varieties of bands}\label{lb}
\end{figure}
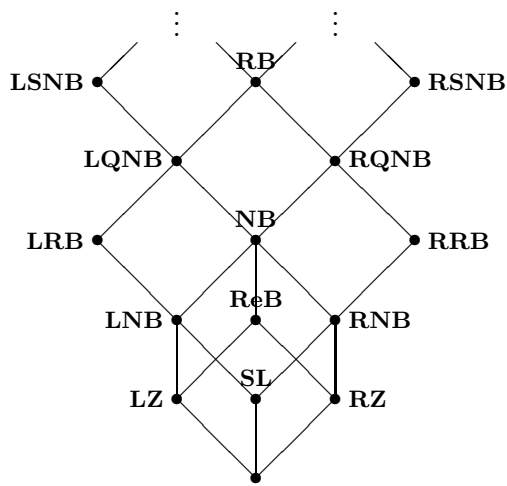

The main result of this note is the following.

\begin{thm}\label{t-main}
Let $\V$ be a variety of bands. Then $\ig{B}$ has all its maximal subgroups free for all $B\in\V$ if and only if $\V$
is contained either in $\mathbf{LSNB}$ or in $\mathbf{RSNB}$.
\end{thm}

This theorem is a direct consequence of the following two propositions.

\begin{pro}\label{p1}
For any left (right) seminormal band $B$, all maximal subgroups of $\ig{B}$ are free.
\end{pro}

\begin{pro}\label{p2}
There exists a regular band $B$ such that $\ig{B}$ has a maximal subgroup isomorphic to $\mathbb{Z}\oplus\mathbb{Z}$.
\end{pro}

The first of these propositions is a generalisation of the well known result of Pastijn \cite[Theorem 6.5]{P2} (cf.\
also \cite{NP,P1}) that all maximal subgroups of $\ig{B}$ are free for any normal band $B$. The other one supplies a
simpler example with the same non-free maximal subgroup than the one considered in \cite[Section 5]{BMM1}. The method
used is the one from \cite{GR}, which is based on the Reidemeister-Schreier type rewriting process for obtaining
presentations of maximal subgroups of semigroups developed in \cite{R-JA}. So, before turning to the proofs of the
above two propositions, we briefly present this general method yielding presentations for maximal subgroups of
$\ig{E}$, $E=E(S)$, for an arbitrary semigroup $S$, and then we explain its particular case when $S$ is a band. Along
the way, we assume some familiarity with the most basic notions of semigroup theory, such as Green's relations and the
structure of bands, see, for example, \cite{Hi,Pe-LinS}.

Let $S$ be a semigroup and let $D$ be a $\mathcal{D}$-class of $S$ containing an idempotent $e_0\in E(S)$. We are going
to label the $\mathcal{R}$-classes contained in $D$ by $R_i$, $i\in I$, while $L_j$, $j\in J$, is the list of all
$\mathcal{L}$-classes of $D$. The $\mathcal{H}$-class $R_i\cap L_j$ will be denoted by $H_{ij}$. Define
$\mathcal{K}=\{(i,j):\ H_{ij}\text{ is a group}\}$; as is well known, $(i,j)\in\mathcal{K}$ if and only if $H_{ij}$
contains an idempotent, which we denote by $e_{ij}$. There is no loss of generality if we assume that both $I$ and $J$
contain an index 1, so that $e_0=e_{11}$.

For a word $\ww\in E\st$, let $\ol\ww$ denote the image of $\ww$ under the canonical monoid homomorphism of $E\st$ into
$S^1$: in other words, when $\ww$ is non-empty, $\ol\ww$ is just the element of $S$ obtained by multiplying in $S$ the
idempotents the concatenation of which is $\ww$. We say that a system of words $\rr_j,\rr'_j\in E\st$, $j\in J$, is a
\emph{Schreier system of representatives} for $D$ if for each $j\in J$:
\begin{itemize}
\item the right multiplications by $\ol{\rr_j}$ and $\ol{\rr'_j}$ are mutually inverse $\mathcal{R}$-class preserving
bijections $L_1\to L_j$ and $L_j\to L_1$, respectively (so, in particular, right multiplication by $\rr_1$ is the
identity mapping on $L_1$);
\item each prefix of $\rr_j$ coincides with $\rr_{j'}$ for some $j'\in J$ (in particular, the empty word is
just $\rr_1$).
\end{itemize}
It is well-known that such a Schreier system always exists. In the following, we assume that one particular Schreier
system has been fixed.

In addition, we will assume that a mapping $i\mapsto j(i)$ has been specified such that $(i,j(i))\in\mathcal{K}$: such
$j(i)$ must exist for each $i\in I$, since $D$ is a regular $\mathcal{D}$-class (as it contains an idempotent), and so
each $\mathcal{R}$-class $R_i$ must contain an idempotent. The index $j(i)\in J$ is called the \emph{anchor} of $R_i$.

Finally, call a \emph{square} a quadruple of idempotents $(e,f,g,h)$ in $D$ such that
$$\begin{array}{ccc}
e & \mathcal{R} & f\\[1mm]
\mathcal{L} && \mathcal{L}\\[1mm]
g & \mathcal{R} & h.
\end{array}$$
Then there are $i,k\in I$ and $j,\ell\in J$ such that $e\in H_{ij}$, $f\in H_{i\ell}$, $g\in H_{kj}$ and $h\in
H_{k\ell}$. For an idempotent $\varepsilon\in S$
we say that it \emph{singularises} the square $(e,f,g,h)$ if any of the following two cases takes place:
\begin{itemize}
\item[(a)] $\varepsilon e=e$ and $\varepsilon g=g$, while $e=f\varepsilon$; or
\item[(b)] $e=\varepsilon g$, along with $e\varepsilon=e$ and $f\varepsilon=f$.
\end{itemize}
Note that case (a) implies $\varepsilon f=f$, $\varepsilon h=h$, $e\varepsilon=e$ and $g=g\varepsilon=h\varepsilon$,
while conditions $\varepsilon e=e$, $f=\varepsilon f=\varepsilon h$, $g\varepsilon=g$ and $h\varepsilon=h$ follow from
(b). The square $(e,f,g,h)$ is \emph{singular} if it is singularised by some idempotent of $S$. Let $\Sigma$ be the set
of all quadruples $(i,k;j,\ell)\in I\times I\times J\times J$ (to be called \emph{singular rectangles}) such that
$(e_{ij},e_{i\ell},e_{kj},e_{k\ell})$ is a singular square in $D$.

The required general result of \cite{GR} can be now paraphrased as follows.

\begin{thm}[Theorem 5 of \cite{GR}]\label{Bob-Nik}
Let $S$ be a semigroup with a non-empty set of idempotents $E=E(S)$. With the notation as above, the maximal subgroup
of the free idempotent generated semigroup $\ig{E}$ containing $e_{11}\in E$ is presented by
$\langle\Gamma\pre\mathfrak{R}\rangle$, where $\Gamma=\{f_{ij}:\ (i,j)\in\mathcal{K}\}$, while $\mathfrak{R}$ consists
of three types of relations:
\begin{itemize}
\item[\rm (i)] $f_{i,j(i)}=1$ for all $i\in I$;
\item[\rm (ii)] $f_{ij}=f_{i\ell}$ for all $i\in I$ and $j,\ell\in J$ such that $\rr_j\cdot e_{i\ell}=\rr_\ell$;
\item[\rm (iii)] $f_{ij}^{-1}f_{i\ell}=f_{kj}^{-1}f_{k\ell}$ for all $(i,k;j,\ell)\in \Sigma$.
\end{itemize}
\end{thm}

For our purpose, we would like to focus on the particular case when $S$ is a band. Then, clearly, $\mathcal{K}=I\times
J$ and $D=\{e_{ij}:\ i\in I,\ j\in J\}$. Since $\mathcal{D}=\mathcal{J}$ in any band, the set of all
$\mathcal{D}$-classes of $B$ is partially ordered; it instantly turns out that, by definition, if $\varepsilon$
singularises a square $(e,f,g,h)$ in $D$, then $D_\varepsilon\geq D$. Now any such $\varepsilon\in B$ induces a pair of
transformations on $I$ and $J$, respectively, in the following sense. For each $i\in I$ and $j\in J$ there are $i',k\in
I$ and $j',\ell\in J$ such that $\varepsilon e_{ij}=e_{i'\ell}$ and $e_{ij}\varepsilon =e_{kj'}$. One immediately sees
that it must be $\ell=j$ and $k=i$, so that $B$ acts on the left on $I$ and on the right on $J$. Thus it is convenient
to write the transformation $\sigma=\sigma_\varepsilon^{(l)}$ induced by $\varepsilon$ on $I$ to the left of its
argument (so that $ee_{ij}=e_{\sigma(i)j}$), while the analogous transformation $\sigma'=\sigma_\varepsilon^{(r)}$ on
$J$ is written to the right (resulting in the rule $e_{ij}e=e_{i(j)\sigma'}$).

\begin{cor}
Let $B$ be a band, let $D$ be a $\mathcal{D}$-class of $B$, and let $e_{11}\in D$. Then the maximal subgroup
$G_{e_{11}}$ of $\ig{B}$ containing $e_{11}$ is presented by $\langle\Gamma\pre\mathfrak{R}\rangle$, where
$\Gamma=\{f_{ij}:\ i\in I, j\in J\}$ and $\mathfrak{R}$ consists of relations
\begin{equation}
f_{i1}=f_{1j}=f_{11}=1\label{rel1}
\end{equation}
for all $i\in I$ and $j\in J$, and
\begin{equation}
f_{ij}^{-1}f_{i\ell} = f_{kj}^{-1}f_{k\ell},\label{rel2}
\end{equation}
where  for some $\varepsilon\in B$ such that $D_\varepsilon\geq D$ the indices $i,k\in I$, $j,\ell\in J$ satisfy one of
the following two conditions:
\begin{itemize}
\item[\rm (a)] $\sigma_\varepsilon^{(l)}(i)=i$, $\sigma_\varepsilon^{(l)}(k)=k$, and $(j)\sigma_\varepsilon^{(r)}=
(\ell)\sigma_\varepsilon^{(r)}=\ell$,
\item[\rm (b)] $\sigma_\varepsilon^{(l)}(i)=\sigma_\varepsilon^{(l)}(k)=k$, $(j)\sigma_\varepsilon^{(r)}=j$ and
$(\ell)\sigma_\varepsilon^{(r)}=\ell$.
\end{itemize}
\end{cor}

\begin{proof}
Since $\mathcal{K}=I\times J$, we have a generator $f_{ij}$ for each $i\in I$ and $j\in J$. Furthermore, the same
reason allows us to choose $j(i)=1$ as the anchor for each $i\in I$. Such a choice will imply that the relations of
type (i) from Theorem \ref{Bob-Nik} take the form $f_{i1}=1$, $i\in I$. In particular, we have $f_{11}=1$. As for the
Schreier system, we can choose $\rr_1$ to be the empty word, $\rr_{j}=e_{1j}$ for all $j\in J\setminus\{1\}$ and
$\rr'_j=e_{11}$ for all $j\in J$. The system $\rr_j$, $j\in J$, of words over $E$ is obviously prefix-closed. Since
$e_{i1}e_{ij}=e_{ij}$ and $e_{ij}e_{11}=e_{i1}$ holds for all $i\in I$, $j\in J$, the right multiplications by $e_{ij}$
and $e_{11}$ are indeed mutually inverse bijections between $L_1$ and $L_j$ and between $L_j$ and $L_1$, respectively.
Hence, the relations of type (ii) reduce to $f_{11}=f_{1j}$, that is, $f_{1j}=1$, for all $j\in J$. Thus we have all
the relations \eqref{rel1}. Finally, the conditions (a) and (b) express precisely the singularisation of a square
$(e_{ij},e_{i\ell},e_{kj},e_{k\ell})$ in $D$ by an element $\varepsilon\in B$; therefore, the relations \eqref{rel2}
correspond to relations of type (iii).
\end{proof}

Rectangles $(i,k;j,\ell)\in I\times J$ of type (a) will be said to be \emph{left-right} singular, while those of type
(b) are \emph{up-down} singular (with respect to $\varepsilon$). Another, more compact way of expressing condition (a)
is $i,k\in\im\sigma_\varepsilon^{(l)}$, $\ell\in\im\sigma_\varepsilon^{(r)}$ and
$(j,\ell)\in\kr\sigma_\varepsilon^{(r)}$, while (b) is equivalent to $k\in\im\sigma_\varepsilon^{(l)}$,
$(i,k)\in\kr\sigma_\varepsilon^{(l)}$ and $j,\ell\in\im\sigma_\varepsilon^{(r)}$.

We can now turn to proving our aforementioned result.

\begin{proof}[Proof of Proposition \ref{p1}]
Without any loss of generality, assume that $B\in\mathbf{RSNB}$ (the case when $B$ belongs to $\mathbf{LSNB}$ is dual).
Recall (e.g.\ from \cite[Proposition II.3.8]{Pe-LinS}) that the variety $\mathbf{RSNB}$ satisfies (and is indeed
defined by) the identity $tuv=tvtuv$. Therefore, if $B=\bigcup_{\alpha\in Y}B_\alpha$ is the greatest semilattice
decomposition of $B$, $a\in B$ and $x,y\in D=B_\alpha$ for some $\alpha\in Y$, then $x=xyx$ and $y=yxy$. Hence, we have
$ax=ax(yx)=ayxaxyx$ and $ay=ay(xy)=axyayxy$, implying $ax\,\mathcal{R}\, ay$. In particular, for any $\varepsilon\in B$
such that $D_\varepsilon\geq D$, $\varepsilon e_{ij}\,\mathcal{R}\, \varepsilon e_{k\ell}$ holds in $D$ for all $i,k\in
I$, $j,\ell\in J$, so the transformation $\sigma_\varepsilon^{(l)}$ is a constant function on $I$.

We conclude that there are no proper (non-degenerate) rectangles $(i,k;j,\ell)$ that are left-right singular with
respect to some $\varepsilon\in B$. In other words, all proper singular rectangles in $I\times J$---and thus all
nontrivial relations of $G_{e_{11}}$---are of the up-down kind:
$$f_{ij}^{-1}f_{i\ell}=f_{k_0j}^{-1}f_{k_0\ell},$$
where $j,\ell$ are two fixed points of $\sigma_\varepsilon^{(r)}$, $i\in I$ is arbitrary, and (since in this context
$\sigma_\varepsilon^{(l)}$ is constant) $\im\sigma_\varepsilon^{(l)}=\{k_0\}$, for some $\varepsilon\in B$. However,
now it is straightforward to deduce the relation \eqref{rel2} for \emph{all} $i,k\in I$ and fixed points $j,\ell$ of
$\sigma_\varepsilon^{(r)}$. Thus we are led to define an equivalence $\theta_B$ of $\bigcup_{\varepsilon\in B,
D_\varepsilon\geq D}\im\sigma_\varepsilon^{(r)}=J$ which is the transitive closure of the relation $\rho_B$ defined by
$(j_1,j_2)\in\rho_B$ if and only if $j_1,j_2\in\im\sigma_\varepsilon^{(r)}$ for some $\varepsilon\in B$. Now it is
almost immediate to see that for all $i,k\in I$ and $j,\ell\in J$ such that $(j,\ell)\in\theta_B$ we have that
$$f_{ij}^{-1}f_{i\ell}=f_{kj}^{-1}f_{k\ell}$$ holds in $G_{e_{11}}$. This immediately implies $f_{k\ell}=1$ for all
$k\in I$ and $\ell\in 1/\theta_B$, as well as
$$f_{kj}=f_{k\ell}$$
for all $k\in I$, whenever $(j,\ell)\in\theta_B$. So, let $j_1=1,j_2\dots,j_m\in J$ be a cross-section of $J/\theta_B$.
Then it is straightforward to eliminate all the relations from the presentation of $G_{e_{11}}$ while reducing its
generating set to $$\{f_{ij_r}:\ i\in I\setminus\{1\},\;2\leq r\leq m\}.$$ In other words, $G_{e_{11}}$ is a free group
of rank $(|I|-1)(m-1)$.
\end{proof}

\begin{proof}[Proof of Proposition \ref{p2}]
Let $B$ be the subband of the free regular band on four generators $a,b,c,d$ consisting of two $\mathcal{D}$-classes: a
$2\times 2$ class $D_1$ consisting of elements $ab,aba,ba,bab$ and a $4\times 4$ class $D_0$ consisting of elements of
the form $\uu_1\vv\uu_2$, where $\uu_1,\uu_2\in\{ab,ba\}$ and $\vv\in\{cd,cdc,dc,dcd\}$. So, we can take
$I=\{abcd,abdc,bacd,badc\}$, the set of all initial parts of words from $D_0$, and $J=\{cdba,dcba,cdab,dcab\}$, the set
of all final parts of those words. A direct computation shows that
\begin{align*}
&\sigma_{ab}^{(l)}=\sigma_{aba}^{(l)}=\left(
\begin{array}{llll}
abcd & abdc& badc & bacd\\
abcd & abdc& abdc & abcd
\end{array}\right),\\
&\sigma_{ba}^{(l)}=\sigma_{bab}^{(l)}=\left(
\begin{array}{llll}
abcd & abdc& badc & bacd\\
bacd & badc& badc & bacd
\end{array}\right),\\
&\sigma_{ab}^{(r)}=\sigma_{bab}^{(r)}=\left(
\begin{array}{llll}
cdba & cdab & dcab & dcba\\
cdab & cdab & dcab & dcab
\end{array}\right),\\
&\sigma_{ba}^{(r)}=\sigma_{aba}^{(r)}=\left(
\begin{array}{llll}
cdba & cdab & dcab & dcba\\
cdba & cdba & dcba & dcba
\end{array}\right).
\end{align*}
If we enumerate (for brevity of further calculations) $abcd\to 1,abdc\to 2,badc\to 3,bacd\to 4$ and $cdba\to 1,cdab\to
2,dcab\to 3,dcba\to 4$, we get
\begin{align*}
&\sigma_{ab}^{(l)}=\sigma_{aba}^{(l)}=\left(
\begin{array}{llll}
1 & 2 & 3 & 4\\
1 & 2 & 2 & 1
\end{array}\right),
&\sigma_{ba}^{(l)}=\sigma_{bab}^{(l)}=\left(
\begin{array}{llll}
1 & 2 & 3 & 4\\
4 & 3 & 3 & 4
\end{array}\right),\\
&\sigma_{ab}^{(r)}=\sigma_{bab}^{(r)}=\left(
\begin{array}{llll}
1 & 2 & 3 & 4\\
2 & 2 & 3 & 3
\end{array}\right),
&\sigma_{ba}^{(r)}=\sigma_{aba}^{(r)}=\left(
\begin{array}{llll}
1 & 2 & 3 & 4\\
1 & 1 & 4 & 4
\end{array}\right).
\end{align*}
Hence, the list of singular rectangles is exhausted by:
\begin{align*}
&(1,2;1,2),(1,2;3,4),(3,4;1,2),(3,4;3,4),\\
&(1,4;2,3),(1,4;1,4),(2,3;2,3),(2,3;1,4).
\end{align*}
This results in $f_{11}=f_{12}=f_{13}=f_{14}=f_{21}=f_{31}=f_{41}=f_{22}=f_{44}=1$ and
$$\begin{array}{lll}
f_{23}=f_{24}, &f_{24}=f_{34}, &f_{43}^{-1}=f_{33}^{-1}f_{34}\\[1mm]
f_{32}=f_{42}, &f_{42}=f_{43}, &f_{23}=f_{32}^{-1}f_{33}.
\end{array}$$

\begin{figure}[th]\centering
\begin{picture}(120.00,132.50)\thinlines
\put(0.00,5.00){\line(0,1){120.00}} \put(30.00,5.00){\line(0,1){120.00}} \put(60.00,5.00){\line(0,1){120.00}}
\put(90.00,5.00){\line(0,1){120.00}} \put(120.00,5.00){\line(0,1){120.00}} \put(0.00,5.00){\line(1,0){120.00}}
\put(0.00,35.00){\line(1,0){120.00}} \put(0.00,65.00){\line(1,0){120.00}} \put(0.00,95.00){\line(1,0){120.00}}
\put(0.00,125.00){\line(1,0){120.00}}

\linethickness{2pt} \put(5.00,9.00){\line(0,1){52.00}} \put(55.00,9.00){\line(0,1){52.00}}
\put(4.00,10.00){\line(1,0){52.00}} \put(4.00,60.00){\line(1,0){52.00}} \put(5.00,69.00){\line(0,1){52.00}}
\put(55.00,69.00){\line(0,1){52.00}} \put(4.00,70.00){\line(1,0){52.00}} \put(4.00,120.00){\line(1,0){52.00}}
\put(65.00,9.00){\line(0,1){52.00}} \put(115.00,9.00){\line(0,1){52.00}} \put(64.00,10.00){\line(1,0){52.00}}
\put(64.00,60.00){\line(1,0){52.00}} \put(65.00,69.00){\line(0,1){52.00}} \put(115.00,69.00){\line(0,1){52.00}}
\put(64.00,70.00){\line(1,0){52.00}} \put(64.00,120.00){\line(1,0){52.00}}

\put(35.00,39.00){\line(0,1){52.00}} \put(85.00,39.00){\line(0,1){52.00}} \put(34.00,40.00){\line(1,0){52.00}}
\put(34.00,90.00){\line(1,0){52.00}} \put(35.00,99.00){\line(0,1){26.00}} \put(85.00,99.00){\line(0,1){26.00}}
\put(34.00,100.00){\line(1,0){52.00}} \put(35.00,5.00){\line(0,1){26.00}} \put(85.00,5.00){\line(0,1){26.00}}
\put(34.00,30.00){\line(1,0){52.00}}

\put(25.00,39.00){\line(0,1){52.00}} \put(0.00,40.00){\line(1,0){26.00}} \put(0.00,90.00){\line(1,0){26.00}}
\put(95.00,39.00){\line(0,1){52.00}} \put(94.00,40.00){\line(1,0){26.00}} \put(94.00,90.00){\line(1,0){26.00}}
\put(0.00,30.00){\line(1,0){26.00}} \put(25.00,5.00){\line(0,1){26.00}} \put(94.00,30.00){\line(1,0){26.00}}
\put(95.00,5.00){\line(0,1){26.00}} \put(0.00,100.00){\line(1,0){26.00}} \put(25.00,99.00){\line(0,1){26.00}}
\put(94.00,100.00){\line(1,0){26.00}} \put(95.00,99.00){\line(0,1){26.00}}

\put(15.00,20.00){\circle*{4}} \put(45.00,20.00){\circle*{4}} \put(75.00,20.00){\circle*{4}}
\put(105.00,20.00){\circle*{4}} \put(15.00,50.00){\circle*{4}} \put(45.00,50.00){\circle*{4}}
\put(75.00,50.00){\circle*{4}} \put(105.00,50.00){\circle*{4}} \put(15.00,80.00){\circle*{4}}
\put(45.00,80.00){\circle*{4}} \put(75.00,80.00){\circle*{4}} \put(105.00,80.00){\circle*{4}}
\put(15.00,110.00){\circle*{4}} \put(45.00,110.00){\circle*{4}} \put(75.00,110.00){\circle*{4}}
\put(105.00,110.00){\circle*{4}}

\put(-5.00,20.00){\makebox(0,0)[rc]{$4$}} \put(-5.00,50.00){\makebox(0,0)[rc]{$3$}}
\put(-5.00,80.00){\makebox(0,0)[rc]{$2$}} \put(-5.00,110.00){\makebox(0,0)[rc]{$1$}}
\put(15.00,130.00){\makebox(0,0)[cb]{$1$}} \put(45.00,130.00){\makebox(0,0)[cb]{$2$}}
\put(75.00,130.00){\makebox(0,0)[cb]{$3$}} \put(105.00,130.00){\makebox(0,0)[cb]{$4$}}
\end{picture}
\caption{The rectangles in $D_0$ singularised by  elements of $D_1$}
\end{figure}
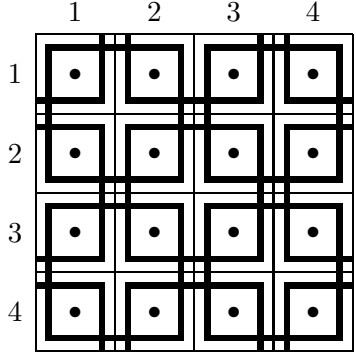

\noindent If we denote $x=f_{23}$ and $y=f_{32}$ we obviously remain with these two generators for $G_{abcdba}$ and a
single relation
$$yx=f_{33}=xy,$$
so $G_{abcdba}\cong \mathbb{Z}\oplus\mathbb{Z}$.
\end{proof}

This completes the proof of Theorem \ref{t-main}.

\begin{rmk}\rm
The band $B$ from the previous proof can be also realised as a regular subband of the free band $FB_3$ on three
generators $a,b,c$ whose elements are from $D'_1=\{ab,aba,ba,bab\}$ and $D'_0=\{\uu c\vv:\ \uu,\vv\in D'_1\}$.
\end{rmk}

We finish the note by several problems that might be subjects of future research in this direction.

\begin{q}\rm
Characterise all bands $B$ with the property that $\ig{B}$ has a non-free maximal subgroup.
\end{q}

\begin{q}\rm
Characterise all groups that arise as maximal subgroups of $\ig{B}$ for some band $B$. The same problem stands for
regular bands $B$, and in fact for $B\in\mathbf{V}$ for any particular band variety $\mathbf{V}\geq\mathbf{RB}$.
\end{q}

\begin{q}\rm
Given a band variety $\mathbf{V}$ and an integer $n\geq 1$, describe the maximal subgroups of
$\ig{\mathfrak{F}_n\mathbf{V}}$, where $\mathfrak{F}_n\mathbf{V}$ denotes the $\mathbf{V}$-free band on a set of $n$
free generators \cite{PeSi}.
\end{q}

\begin{ack}\rm
The author is grateful to the anonymous referee, whose careful reading, comments and suggestions significantly improved
the presentation of the results.
\end{ack}


\end{document}